\documentclass[a4paper,11pt]{amsart}
\pdfoutput=1

\usepackage[ansinew]{inputenc}
\usepackage[danish, english]{babel}
\usepackage{amssymb}
\usepackage{amsmath}
\usepackage{amsthm}
\usepackage{ifthen}
\usepackage{mathtools}

\usepackage{tikz}
\usepackage{todonotes}

\usepackage[final,pdfauthor={Sune Precht Reeh},pdftitle={The abelian monoid of fusion-stable finite sets is free}]{hyperref}
\usepackage{bookmark}
\usepackage[abbrev,msc-links]{amsrefs}

\numberwithin{equation}{section}
\numberwithin{figure}{section}
\numberwithin{table}{section}
\theoremstyle{plain}
\newtheorem{mainthmIntro}{Theorem}
\newtheorem{mainthmIntroGroup}{Theorem}
\newtheorem{mainthm}{Theorem}

\newtheorem{thm}{Theorem}[section]
\newtheorem{prop}[thm]{Proposition}
\newtheorem{lem}[thm]{Lemma}
\newtheorem{cor}[thm]{Corollary}
\theoremstyle{definition}
\newtheorem{dfn}[thm]{Definition}
\newtheorem{ex}[thm]{Example}

\DeclareMathOperator{\Span}{Span}
\DeclareMathOperator{\Mor}{Mor}
\DeclareMathOperator{\Hom}{Hom}
\DeclareMathOperator{\Aut}{Aut}
\DeclareMathOperator{\Rep}{Rep}
\DeclareMathOperator{\im}{im}
\DeclareMathOperator{\coker}{coker}

\newcommand{\xto}[0]{\xrightarrow}
\newcommand{\into}[0]{\hookrightarrow}
\newcommand{\x}[0]{\times}

\newcommand{\Ø}[0]{\emptyset}
\newcommand{\ph}[0]{\varphi}

\newcommand{\Z}[0]{\mathbb Z}
\newcommand{\F}[0]{\mathbb F}
\newcommand{\C}[0]{\mathbb C}

\newcommand{\PGL}[0]{\mathrm{PGL}}
\newcommand{\D}[0]{\mathrm{D}}
\newcommand{\SD}[0]{\mathrm{SD}}

\DeclarePairedDelimiter{\abs}\lvert\rvert
\DeclarePairedDelimiter{\gen}\langle\rangle

\newcommand{\cal}{\mathcal}
\renewcommand{\tilde}{\widetilde}

\renewcommand{\bar}{\overline}

\newcommand{\ccset}[1]{Cl(#1)}
\newcommand{\free}[1]{{\protect\tilde\Omega(#1)}}

\newcommand{\lc}[1]{\prescript{#1\!}{}}

\usetikzlibrary{arrows,matrix,decorations,positioning,calc}

\renewcommand{\theenumi}{(\roman{enumi})}\renewcommand{\labelenumi}{\theenumi}

\title{The abelian monoid of fusion-stable finite sets is free}
\author[S. P. Reeh]{Sune Precht Reeh}
\address{Department of Mathematics,  Massachusetts Institute of Technology, Cambridge, Massachusetts, USA}
\email{reeh@mit.edu}
\subjclass[2010]{}

\thanks{Supported by the Danish National Research Foundation through the Centre for Symmetry and Deformation (DNRF92), and by The Danish Council for Independent Research's Sapere Aude program (DFF – 4002-00224).}

\begin{document}
\begin{abstract}
We show that the abelian monoid of isomorphism classes of $G$-stable finite $S$-sets is free for a finite group $G$ with Sylow $p$-subgroup $S$; here a finite $S$-set is called $G$-stable if it has isomorphic restrictions to $G$-conjugate subgroups of $S$. These $G$-stable $S$-sets are of interest, e.g., in homotopy theory. We prove freeness by constructing an explicit (but somewhat non-obvious) basis, whose elements are in one-to-one correspondence with the $G$-conjugacy classes of subgroups in $S$.
As a central tool of independent interest, we give a detailed description of the embedding of the Burnside ring for a saturated fusion system into its associated ghost ring.
\end{abstract}

\maketitle

\section{Introduction}
Finite $G$-sets, where $G$ is a finite group, appear again and again throughout mathematics, e.g., in homotopy theory. In certain instances we are however interested, not in the $G$-sets themselves, but instead in the shadows cast by $G$-sets when we restrict the actions to a Sylow $p$-subgroup $S$ of $G$. When a finite set $X$ has an action of $S$ that ``looks like'' it comes from a $G$-action, we say that the $S$-set $X$ is $G$-stable (see below). $G$-stable $S$-sets occur for instance in homotopy theory when describing maps between classifying spaces. The isomorphism classes of these $G$-stable $S$-sets together form an abelian monoid with disjoint union as the addition.
In this paper we construct a basis for the abelian monoid of $G$-stable sets when $G$ is a finite group with Sylow $p$-subgroup $S$. Theorem \ref{thmMonoidBasisIntroGroup} states that this abelian monoid is free, and that the basis elements are in one-to-one correspondence with the $G$-conjugacy classes of subgroups in $S$.
Theorem \ref{thmMonoidBasisIntroGroup} is a special case of the more general Theorem \ref{thmMonoidBasisIntro} formulated for a saturated fusion systems $\cal F$ over $S$.
As a main tool for proving Theorem \ref{thmMonoidBasisIntro} we describe the Burnside ring of a saturated fusion system $\cal F$, and its embedding into the associated ghost ring (Theorem \ref{thmYoshidaFusionIntro}).

In more detail, let us consider a finite group $G$ acting on a finite set $X$. We can restrict the action to a Sylow $p$-subgroup $S$ of $G$.
The resulting $S$-set has the property that it stays the same (up to $S$-isomorphism) whenever we change the action via a conjugation map from $G$. More precisely, if $P\leq S$ is a subgroup and $\ph\colon P\to S$ is a homomorphism given by conjugation with some element of $G$, we can turn $X$ into a $P$-set by using $\ph$ to define the action $p.x:=\ph(p)x$. We denote the resulting $P$-set by $\prescript{}{P,\ph}X$. In particular when $incl\colon P\to S$ is the inclusion map, $\prescript{}{P,incl}X$ has the usual restricted action from $S$ to $P$.
When a finite $S$-set $X$ is the restriction of a $G$-set, then $X$ has the property
\begin{equation}\label{charGstable}
\parbox[c]{.9\textwidth}{\emph{$\prescript{}{P,\ph}{X}$ is isomorphic to $\prescript{}{P,incl}X$ as $P$-sets, for all $P\leq S$ and homomorphisms $\ph\colon P\to S$ induced by $G$-conjugation.}}
\end{equation}
Any $S$-set with property \eqref{charGstable} is called \emph{$G$-stable}. Whenever we restrict a $G$-set to $S$, the resulting $S$-set is $G$-stable; however there are $G$-stable $S$-sets whose $S$-actions do not extend to actions of $G$.

The isomorphism classes of finite $S$-sets form a semiring $A_+(S)$ with disjoint union as addition and cartesian product as multiplication. The collection of $G$-stable $S$-sets is closed under addition and multiplication, hence $G$-stable sets form a subsemiring.

\begin{mainthmIntroGroup}\label{thmMonoidBasisIntroGroup}
Let $G$ be a finite group with Sylow $p$-group $S$.
Every $G$-stable $S$-set splits uniquely, up to $S$-isomorphism, as a disjoint union of irreducible $G$-stable sets, and there is a one-to-one correspondence between the irreducible $G$-stable sets and $G$-conjugacy classes of subgroups in $S$.

Hence the semiring of $G$-stable $S$-sets is additively a free commutative monoid with rank equal to the number of $G$-conjugacy classes of subgroups in $S$.
\end{mainthmIntroGroup}
As part of the proof we give an explicit construction of the irreducible $G$-stable sets (see Proposition \ref{propIrreducibleSets}).

It is a well-known fact that any finite $S$-set splits uniquely into orbits (i.e. transitive $S$-sets), and the isomorphism type of a transitive set $S/P$ depends only on the subgroup $P$ up to $S$-conjugation. Theorem \ref{thmMonoidBasisIntroGroup} states that this fact generalizes nicely to $G$-stable $S$-sets, which turns out to be less obvious than it might first appear.

If we consider $G$-sets and restrict their actions to $S$, then two non-isomorphic $G$-sets might very well become isomorphic as $S$-sets. Therefore even though finite $G$-sets decompose uniquely into orbits, we have no guarantee that this decomposition remains unique when we restrict the actions to the Sylow subgroup $S$. In fact, uniqueness of decompositions fails in general when we consider restrictions of $G$-sets to $S$, as demonstrated in Example \ref{exNonUnique} for the symmetric group $\Sigma_5$ and a Sylow 2-subgroup.
It is therefore perhaps a surprise that if we consider \emph{all} $G$-stable $S$-sets, and not just the restrictions of actual $G$-sets, we are again able to write stable sets as a disjoint union of irreducibles in a unique way.

It is also worth noting that the analogue of Theorem \ref{thmMonoidBasisIntroGroup} is false if we consider representations instead of sets: The submonoid of $G$-stable $S$-representations is not a free submonoid of the free monoid of complex $S$-representations when $G=\PGL_3(\F_3)$ and $p=2$. We explain this counterexample in Appendix \ref{secRepresentations}, in particular in Example \ref{exCounterexampleRepr}.

\medskip
The proof of Theorem \ref{thmMonoidBasisIntroGroup} relies only on the way $G$ acts on the subgroups of $S$ by conjugation. We therefore state and prove the theorem in general for abstract saturated fusion systems, which abstractly model the conjugacy relations within a $p$-group induced by an ambient group (see Definitions \ref{dfnFusSys} and \ref{dfnSatFusSys}).

If $\cal F$ is a fusion system over a $p$-group $S$, we say that an $S$-set $X$ is \emph{$\cal F$-stable} if it satisfies
\begin{equation}\label{charFstable}
\parbox[c]{.9\textwidth}{\emph{$\prescript{}{P,\ph}X$ is isomorphic to $\prescript{}{P,incl}X$ as $P$-sets, for all $P\leq S$ and homomorphisms $\ph\colon P\to S$ in $\cal F$.}}
\end{equation}
The $\cal F$-stable $S$-sets form a semiring $A_+(\cal F)$ since the disjoint union and cartesian product of $\cal F$-stable sets is again $\cal F$-stable. Theorem \ref{thmMonoidBasisIntroGroup} then generalizes to Theorem \ref{thmMonoidBasisIntro} below, which we prove instead.
\begin{mainthmIntro}\label{thmMonoidBasisIntro}
Let $\cal F$ be a saturated fusion system over a $p$-group $S$.
Every $\cal F$-stable $S$-set splits uniquely, up to $S$-isomorphism, as a disjoint union of irreducible $\cal F$-stable sets, and there is a one-to-one correspondence between the irreducible $\cal F$-stable sets and conjugacy/isomorphism classes of subgroups in the fusion system $\cal F$.

Hence the semiring $A_+(\cal F)$ of $\cal F$-stable $S$-sets is additively a free commutative monoid with rank equal to the number of conjugacy classes of subgroups in $\cal F$.
\end{mainthmIntro}

One reason for interest in Theorem \ref{thmMonoidBasisIntro} is from homotopy theory, where classifying spaces for groups and maps between them play an important role. For finite groups $G,H$, or in general discrete groups, the homotopy classes of unbased maps $[BG,BH]$ is in bijection with $\Rep(G,H)= H\backslash\Hom(G,H)$, where $H$ acts on $\Hom(G,H)$ by post-conjugation. Hence $[BG,B\Sigma_n]$ corresponds to the different way $G$ can act on a set with $n$ elements up to $G$-isomorphism. This implies that for a finite group $G$ we have $[BG,\coprod_n B\Sigma_n]\cong A_+(G)$ as monoids.

However in homotopy theory one is often only interested in studying classifying spaces, and maps
between them, one prime at a time via the Bousfield-Kan $p$-completion functor $(-)^{\wedge}_p$, see Sections Sections I.1, VI.6 and VII.5 of \cite{BousfieldKan}.
In this context, when $S$ is a $p$-group, a formula of Mislin \cite{Mislin}*{Formula 4} says that $S$ satisfies \[[BS,\coprod_n (B\Sigma_n)^\wedge_p]\simeq A_+(S)\] as monoids. (See also \citelist{\cite{DwyerZabrodsky}\cite{Lannes}\cite{Miller}\cite{CarlssonSullivan}} which Mislin's work builds upon.)
For a general finite group $G$, the monoid $[BG,\coprod_n (B\Sigma_n)^\wedge_p]$ is highly interesting but still mysterious.
Restriction along the inclusion $\iota\colon S\to G$ of a Sylow $p$-subgroup induces a map
\[\iota^*\colon [BG,\coprod_n (B\Sigma_n)^\wedge_p] \to [BS,\coprod_n (B\Sigma_n)^\wedge_p] \simeq A_+(S),\]
and the image must necessarily be contained in the collection of $G$-stable sets $A_+(\cal F_S(G))$, where $\cal F_S(G)$ is the fusion system over $S$ generated by $G$.

The map $\iota^*\colon [BG,\coprod_n (B\Sigma_n)^\wedge_p] \to A_+(\cal F_S(G))$ is an isomorphism in the cases where both the left hand side and the right hand side have been calculated, though both
injectivity and surjectivity is currently unknown in general. In any
case, Theorem \ref{thmMonoidBasisIntro} shows that the algebraic approximation $A_+(\cal F_S(G))$ has a
very regular structure for any finite group $G$, and hence provides
information in understanding the monoid on the left-hand side.

\medskip
An important tool in proving Theorem \ref{thmMonoidBasisIntro} is \emph{the Burnside ring of $\cal F$}, denoted by $A(\cal F)$. We can either define $A(\cal F)$ as the Grothendieck group of the semiring $A_+(\cal F)$ of $\cal F$-stable sets, or we can define $A(\cal F)$ as the subring of $A(S)$ consisting of all $\cal F$-stable elements, where the $\cal F$-stable elements satisfy a property similar to \eqref{charFstable}. Thanks to Proposition \ref{propGrothendieckFusion} we know that these two definitions coincide for saturated fusion systems.

We note that there is an earlier definition by Diaz-Libman in \cite{DiazLibman} of a Burnside ring for $\cal F$, only involving the so-called $\cal F$-centric subgroups of $S$, while the Burnside ring defined here concerns \emph{all} subgroups of $S$ in relation to the fusion system $\cal F$. More precisely, after $p$-localization the Diaz-Libman \emph{centric} Burnside ring for $\cal F$ is the quotient of the Burnside ring $A(\cal F)$ defined here by the non-centric part of $A(\cal F)$, as described in \cite{DiazLibman:Segal}*{Theorem A} and even further in \cite{ReehIdempotent}*{Proposition 4.8}.

\medskip
The Burnside ring of $\cal F$ inherits the homomorphism of marks $\Phi$ from $A(S)$ by restriction, embedding $A(\cal F)$ into a product of a suitable number of copies of $\Z$. As a main step in proving Theorem \ref{thmMonoidBasisIntro}, we show that this mark homomorphism has properties analogous the mark homomorphism for groups:
\begin{mainthmIntro}\label{thmYoshidaFusionIntro}
Let $\cal F$ be a saturated fusion system over a $p$-group $S$, and let $A(\cal F)$ be the Burnside ring of $\cal F$, i.e. the subring consisting of the $\cal F$-stable elements in the Burnside ring of $S$.
Then there is a ring homomorphism $\Phi$ and a group homomorphism $\Psi$ that fit together in the following short-exact sequence of groups:
\[0 \to A(\cal F) \xto{\Phi} \prod_{\substack{\text{conj. classes of}\\ \text{subgroups in $\cal F$}}} \Z\ \xto{\Psi} \prod_{\substack{\text{$[P]_{\cal F}$ conj. class of}\\\text{subgroups in $\cal F$}}} \Z / \abs{W_S P}\Z\ \to 0,\]
where $P\leq S$ is a fully $\cal F$-normalized representative of $[P]_{\cal F}$, and $W_S P := N_S P / P$.\enlargethispage{.5cm}

The map $\Phi$ comes from restricting the mark homomorphism of $A(S)$, and $\Psi$ is given by the $[P]_{\cal F}$-coordinate functions
\[\Psi_{P}(f)  := \sum_{\bar s\in W_S P} f_{\gen s P} \pmod {\abs{W_S P}}\]
when $P$ is a fully normalized representative of the conjugacy class $[P]_{\cal F}$ of subgroups in $\cal F$.
Here $\Psi_P=\Psi_{P'}$ if $P\sim_{\cal F} P'$ are both fully normalized.
\end{mainthmIntro}
Theorem \ref{thmYoshidaFusionIntro} generalizes previous results by Burnside, Dress and others (see \cite{Dress}, \cite{tomDieck}*{Section 1} or \cite{YoshidaSES}) concerning the mark homomorphism and congruence relations for Burnside rings of finite groups, and such congruence relations can for instance be used, in \cite{ReehIdempotent}*{Corollary 6.6}, for determining idempotents in $A(\cal F)_{(p)}$. As with Theorem \ref{thmMonoidBasisIntro}, Theorem \ref{thmYoshidaFusionIntro} is interesting from the viewpoint of homotopy theory: Grodal has recently announced \citelist{\cite{GrodalOberwolfach}\cite{GrodalBurnside}} that the map on Grothendieck groups
\[Gr([BG,\coprod_n (B\Sigma_n)^\wedge_p]) \to A(\cal F_S(G))\]
\emph{is} an isomorphism, so Theorem \ref{thmYoshidaFusionIntro} also provides information about a homotopical object via this map.

\subsection*{Acknowledgements}
The claim of Theorem \ref{thmMonoidBasisIntro} was initially suggested by Matthew Gelvin and formed by his previous interest in similar problems as well as my own work on my masters thesis and the first part of \cite{ReehIdempotent}. Matthew's claim became the subject of
conversations at the University of Copenhagen, involving also Jesper Michael Møller, Bob Oliver,
Kasper Andersen and Jesper Grodal. Due to a seeming lack of systematic structure in the irreducible stable sets, and the fact that the decompositions are only unique up to $S$-isomorphism, the claim was met with initial skepticism, but Kasper Andersen produced a large amount of computer evidence supporting the claim's
validity (using Magma). I would like to thank them all. In particular it was Kasper Andersen's examples that gave me the idea for Lemma \ref{lemMoreFixpoints}, which formed the missing link in the proof of Theorem \ref{thmMonoidBasisIntro}.
I also thank my PhD advisor Jesper Grodal for his helpful suggestions and feedback during the writing of this paper.

\section{Fusion systems}\label{secFusionSystems}
The next few pages contain a very short introduction to fusion systems. The aim is to introduce the terminology from the theory of fusion systems that will be used in the paper, and to establish the relevant notation. For a proper introduction to fusion systems see for instance Part I of ``Fusion Systems in Algebra and Topology'' by Aschbacher, Kessar and Oliver, \cite{AKO}.

\begin{dfn}\label{dfnFusSys}
A \emph{fusion system} $\cal F$ over a $p$-group $S$, is a category where the objects are the subgroups of $S$, and for all $P,Q\leq S$ the morphisms must satisfy:
\begin{enumerate}
\item Every morphism $\ph\in \Mor_{\cal F}(P,Q)$ is an injective group homomorphism, and the composition of morphisms in $\cal F$ is just composition of group homomorphisms.
\item $\Hom_{S}(P,Q)\subseteq \Mor_{\cal F}(P,Q)$, where
    \[\Hom_S(P,Q) = \{c_s \mid s\in N_S(P,Q)\}\]
    is the set of group homomorphisms $P\to Q$ induced by $S$-conjugation.
\item For every morphism $\ph\in \Mor_{\cal F}(P,Q)$, the group isomorphisms $\ph\colon P\to \ph P$ and $\ph^{-1}\colon \ph P \to P$ are elements of $\Mor_{\cal F}(P,\ph P)$ and $\Mor_{\cal F}(\ph P, P)$ respectively.
\end{enumerate}
We also write $\Hom_{\cal F}(P,Q)$ or just $\cal F(P,Q)$ for the morphism set $\Mor_{\cal F}(P,Q)$; and the group $\cal F(P,P)$ of automorphisms is denoted by $\Aut_{\cal F}(P)$.
\end{dfn}

The canonical example of a fusion system comes from a finite group $G$ with a given $p$-subgroup $S$.
The fusion system of $G$ over $S$, denoted $\cal F_S(G)$, is the fusion system over $S$ where the morphisms from $P\leq S$ to $Q\leq S$ are the homomorphisms induced by $G$-conjugation:
\[\Hom_{\cal F_S(G)}(P,Q) := \Hom_G(P,Q) = \{c_g \mid g\in N_G(P,Q)\}.\]
A particular case is the fusion system $\cal F_S(S)$ consisting only of the homomorphisms induced by $S$-conjugation.

Let $\cal F$ be an abstract fusion system over $S$. We say that two subgroup $P,Q\leq S$ are \emph{$\cal F$-conjugate}, written $P\sim_{\cal F} Q$, if they a isomorphic in $\cal F$, i.e. there exists a group isomorphism $\ph\in \cal F(P,Q)$.
The relation of $\cal F$-conjugation is an equivalence relation, and the set of $\cal F$-conjugates to $P$ is denoted by $[P]_{\cal F}$. The set of all $\cal F$-conjugacy classes of subgroups in $S$ is denoted by $\ccset{\cal F}$.
Similarly, we write $P\sim_S Q$ if $P$ and $Q$ are $S$-conjugate, the $S$-conjugacy class of $P$ is written $[P]_{S}$ or just $[P]$, and we write $\ccset S$ for the set of $S$-conjugacy classes of subgroups in $S$.
Since all $S$-conjugation maps are in $\cal F$, any $\cal F$-conjugacy class $[P]_{\cal F}$ can be partitioned into disjoint $S$-conjugacy classes of subgroups $Q\in [P]_{\cal F}$.

We say that $Q$ is \emph{$\cal F$-} or \emph{$S$-subconjugate} to $P$ if $Q$ is respectively $\cal F$- or $S$-conjugate to a subgroup of $P$, and we denote this by $Q\lesssim_{\cal F} P$ or $Q\lesssim_S P$ respectively.
In the case where $\cal F = \cal F_S(G)$, we have $Q\lesssim_{\cal F} P$ if and only if $Q$ is $G$-conjugate to a subgroup of $P$; and the $\cal F$-conjugates of $P$, are just those $G$-conjugates of $P$ which are contained in $S$.

A subgroup $P\leq S$ is said to be \emph{fully $\cal F$-normalized} if $\abs{N_S P} \geq \abs{N_S Q}$ for all $Q\in [P]_{\cal F}$; and similarly $P$ is \emph{fully $\cal F$-centralized} if $\abs{C_S P} \geq \abs{C_S Q}$ for all $Q\in [P]_{\cal F}$.

\begin{dfn}\label{dfnSatFusSys}
A fusion system $\cal F$ over $S$ is said to be \emph{saturated} if the following properties are satisfied for all $P\leq S$:
\begin{enumerate}
\item If $P$ is fully $\cal F$-normalized, then $P$ is fully $\cal F$-centralized, and $\Aut_S(P)$ is a Sylow $p$-subgroup of $\Aut_{\cal F}(P))$.
\item Every homomorphism $\ph\in \cal  F(P,S)$ where $\ph (P)$ is fully $\cal F$-centralized, extends to a homomorphism $\ph\in \cal F(N_\ph,S)$ where
    \[N_\ph:= \{x\in N_S(P) \mid \exists y\in S\colon \ph\circ c_x = c_y\circ \ph\}.\]
\end{enumerate}
\end{dfn}
The saturated fusion systems form a class of particularly nice fusion systems, and the saturation axiom are a way to emulate the Sylow theorems for finite groups. In particular, whenever $S$ is a Sylow $p$-subgroup of $G$, then the Sylow theorems imply that the induced fusion system $\cal F_S(G)$ is saturated (see e.g. \cite{AKO}*{Theorem 2.3}).

In this paper, we shall rarely use the defining properties of saturated fusion systems directly. We shall instead mainly use the following lifting property that saturated fusion systems satisfy:
\begin{lem}[{\cite{RobertsShpectorov}}]\label{lemNormalizerMap}
Let $\cal F$ be saturated. Suppose that $P\leq S$ is fully normalized, then for each $Q\in [P]_{\cal F}$ there exists a homomorphism $\ph\in \cal F(N_S Q, N_S P)$ with $\ph(Q) = P$.
\end{lem}
\noindent For the proof, see Lemma 4.5 of \cite{RobertsShpectorov} or Lemma 2.6(c) of \cite{AKO}.

\section{Burnside rings for groups}\label{secBurnside}
In this section we consider the Burnside ring of a finite group $S$, and the semiring of finite $S$-sets. We recall the structure of the Burnside ring $A(S)$ and how to describe the elements and operations of $A(S)$ in terms of fixed points and the homomorphism of marks. In this section $S$ can be any finite group, but later we shall only need the case where $S$ is a $p$-group.

We consider finite $S$-sets up to $S$-isomorphism, and let $A_+(S)$ denote the set of isomorphism classes.
Given a finite $S$-set $X$, we denote the isomorphism class of $X$ by $[X]\in A_+(S)$.
Taking disjoint union as addition and cartesian product as multiplication gives a commutative semiring structure on $A_+(S)$.
Additively $A_+(S)$ is a free commutative monoid, where the basis consists of the (isomorphism classes of) transitive $S$ sets, i.e. $[S/P]$ where $P$ is a subgroup of $S$. Two transitive $S$-sets $S/P$ and $S/Q$ are isomorphic if and only if $P$ is conjugate to $Q$ in $S$.

To describe the multiplication of the semiring $A_+(S)$, it is enough to know the products of basis elements $[S/P]$ and $[S/Q]$. By taking the product $(S/P)\x (S/Q)$ and considering how it breaks into orbits, one reaches the following double coset formula for the multiplication in $A_+(S)$:
\begin{equation}\label{eqSingleBurnsideDoubleCoset}
[S/P]\cdot [S/Q] = \sum_{s \in [P\backslash S /Q]} [S/(P\cap \lc s Q)],
\end{equation}
where $[P\backslash S /Q]$ is a set of representative of the double cosets $PsQ$ with $s\in S$.

The \emph{Burnside ring of $S$}, denoted $A(S)$, is constructed as the Grothendieck group of $A_+(S)$, consisting of formal differences of finite $S$-sets. Additively, $A(S)$ is a free abelian group with the same basis as $A_+(S)$.
For each element $X\in A(S)$ we define $c_P(X)$, with $P\leq S$, to be the coefficients when we write $X$ as a linear combination of the basis elements $[S/P]$ in $A(S)$, i.e.
\[X= \sum_{[P]\in \ccset S} c_P(X) \cdot [S/P].\]
Where $\ccset S$ denotes the set of $S$-conjugacy classes of subgroup in $S$.

The resulting maps $c_P\colon A(S) \to \Z$ are group homomorphisms, but they are \emph{not} ring homomorphisms.
Note also that an element $X$ is in $A_+(S)$, i.e. $X$ is an $S$-set, if and only if $c_P(X)\geq 0$ for all $P\leq S$.

Instead of counting orbits, an alternative way of characterising an $S$-set is counting the fixed points for each subgroup $P\leq S$. For every $P\leq S$ and $S$-set $X$, we denote the number of fixed points by $\Phi_{P}(X) := \abs*{X^P}$, and this number only depends on $P$ up to $S$-conjugation.
Since we have
\[\abs*{(X \sqcup Y)^P}= \abs*{X^P}+\abs* {Y^P}, \quad\text{and}\quad\abs*{(X \x Y)^P}= \abs*{X^P}\cdot\abs*{Y^P}\] for all $S$-sets $X$ and $Y$, the \emph{fixed point map} $\Phi_{P}\colon A_+(S)\to \Z$ extends to a ring homomorphism $\Phi_{P}\colon A(S) \to \Z$.
On the basis elements $[S/P]$, the number of fixed points is given by
\[\Phi_{Q}([S/P]) = \abs*{(S/P)^Q}= \frac{\abs{N_S(Q,P)}}{\abs P},\]
where $N_S(Q,P) = \{s\in S \mid \lc s Q \leq P\}$ is the transporter in $S$ from $Q$ to $P$.
In particular, $\Phi_{Q}([S/P])\neq 0$ if and only if $Q\lesssim_S P$ ($Q$ is conjugate to a subgroup of $P$).

\phantomsection\label{dfnFreeS}\label{dfnObsS}
We have one fixed point homomorphism $\Phi_P$ per conjugacy class of subgroups in $S$, and we combine them into the \emph{homomorphism of marks} $\Phi=\Phi^S\colon A(S) \xto{\prod_{[P]} \Phi_{P}} \prod_{[P]\in \ccset S} \Z$. This ring homomorphism maps $A(S)$ into the product ring $\free{S}:=\prod_{[P]\in \ccset S} \Z$ which is the so-called \emph{ghost ring} for the Burnside ring $A(S)$.

Results by Burnside, Dress and others show that the mark homomorphism is injective, and that the cokernel of $\Phi$ is the  \emph{obstruction group} $Obs(S) := \prod_{[P]\in \ccset S} (\Z/\abs{W_S P}\Z)$ -- where $W_S P := N_S P / P$. These statements are combined in the following proposition, the proof of which can be found in \cite{Dress}, \cite{tomDieck}*{Chapter 1} and \cite{YoshidaSES}*{Lemma 2.1}.
\begin{prop}\label{propYoshidaGroup}
Let $\Psi=\Psi^S\colon \free S \to Obs(S)$ be given by the $[P]$-coordinate functions
\[\Psi_{P}(\xi)  := \sum_{\bar s\in W_S P} \xi_{\gen s P} \pmod {\abs{W_S P}}.\]
Here $\xi_{\gen s P}$ denotes the $[\gen s P]$-coordinate of an element $\xi\in \free S = \prod_{[P]\in \ccset S} \Z$.

The following sequence of abelian groups is then exact:
\[0\to A(S) \xto{\Phi} \free S \xto{\Psi} Obs(S) \to 0.\]
Moreover, $\Phi$ is a ring homomorphism, while $\Psi$ is just a group homomorphism.
\end{prop}
The strength of this result is that it enables one to perform calculations for the Burnside ring $A(S)$ inside the much nicer product ring $\free S$, where we identify each element $X\in A(S)$ with its fixed point vector $(\Phi_Q(X))_{[Q]\in \ccset S}$.

\begin{cor}\label{corYoshidaGroup}
For a normal subgroup $P\leq S$, and an $S$-set $X$, we have
\[\sum_{\bar s\in S/P} \Phi_{\gen{s}P}(X) \equiv 0 \pmod{\abs{S/P}}.\]
\end{cor}
\begin{proof}
  Applying Proposition \ref{propYoshidaGroup} with $W_S P = S/P$, gives $\Psi_P(\Phi(X))=0$ in $\Z / \abs{P/S}\Z$.
\end{proof}

\section{Stable sets for a fusion system}
Let $\cal F$ be a fusion system over a $p$-group $S$. In this section we rephrase the property of $\cal F$-stability in terms of the fixed point homomorphisms, and show in Example \ref{exNonUnique} how Theorem \ref{thmMonoidBasisIntro} can fail for a group $G$ if we only consider $S$-sets that are restrictions of $G$-sets, instead of considering all $G$-stable sets. We also consider two possible definitions for the Burnside ring of a fusion system -- these agree if $\cal F$ is saturated.
The proof of Theorem \ref{thmMonoidBasisIntro} begins in section \ref{secProofOfTheorem} in earnest.

A finite $S$-set $X$ is said to be \emph{$\cal F$-stable} if it satisfies \eqref{charFstable}:
\begin{equation*}
\parbox[c]{.9\textwidth}{\emph{$\prescript{}{P,\ph}X$ is isomorphic to $\prescript{}{P,incl}X$ as $P$-sets, for all $P\leq S$ and homomorphisms $\ph\colon P\to S$ in $\cal F$.}}
\end{equation*}
In order to define $\cal F$-stability not just for $S$-sets, but for all elements of the Burnside ring, we extend $\prescript{}{P,\ph}X$ to all $X\in A(S)$. Given a homomorphism $\ph\in \cal F(P,S)$ and an $S$-set $X$, the $P$-set $\prescript{}{P,\ph}X$ was defined as $X$ with the action restricted along $\ph$, that is $p.x:=\ph(p)x$ for $x\in X$ and $p\in P$. This construction then extends linearly to a ring homomorphism $r_{\ph}\colon A(S) \to A(P)$, and we denote $\prescript{}{P,\ph}X:= r_\ph(X)$ for all $X\in A(S)$. In this way \eqref{charFstable} makes sense for all $X\in A(S)$.

It is possible to state $\cal F$-stability purely in terms of fixed points and the homomorphism of marks for $A(S)$. The following lemma seems to be generally known, but not published anywhere, so we include it for the sake of completeness. A version of this lemma was included in the PhD thesis of Gelvin as \cite{Gelvin}*{Proposition 3.2.3}, and previously the lemma has least been implicitely used by Broto-Levi-Oliver in the proof of \cite{BrotoLeviOliver}*{Proposition 5.5}. A special case of the lemma, for bisets, was also proved by Ragnarsson-Stancu as part (b) and (c) of \cite{RagnarssonStancu}*{Lemma 4.8}.
\begin{lem}[\cite{Gelvin}]
The following are equivalent for all elements $X\in A(S)$:
\begin{enumerate}
\item\label{itemPhiBurnsideEq}  $\prescript{}{P,\ph} X = \prescript{}{P,incl} X$ in $A(P)$ for all $\ph\in \cal F(P,S)$ and $P\leq S$.
\item\label{itemPhiStable} $\Phi_{P}(X) = \Phi_{\ph P}(X)$ for all $\ph\in \cal F(P,S)$ and $P\leq S$.
\item\label{itemFConjStable} $\Phi_{P}(X) = \Phi_{Q}(X)$ for all pairs $P,Q\leq S$ with $P\sim_{\cal F} Q$.
\end{enumerate}
\end{lem}
We shall primarily use \ref{itemPhiStable} and \ref{itemFConjStable} to characterize $\cal F$-stability.

\begin{proof}
Let $\Phi^P\colon A(P)\to \free P$ be the homomorphism of marks for $A(P)$, and note that $\Phi^P_{R}(\prescript{}{P,incl} X) = \Phi_{R}(X)$ for all $R\leq P\leq S$.

By the definition of the $P$-action on $\prescript{}{P,\ph} X$, we have $(\prescript{}{P,\ph} X)^R = X^{\ph R}$ for any $S$-set $X$ and all subgroups $R\leq P$. This generalizes to
\[\Phi^P_{R}(\prescript{}{P,\ph} X) = \Phi_{\ph R}(X)\]
for $X\in A(S)$.

Assume \ref{itemPhiBurnsideEq}. Then we immediately get
\[\Phi_{P}(X) =\Phi^P_{P}(\prescript{}{P,incl} X) = \Phi^P_{P}(\prescript{}{P,\ph} X) = \Phi_{\ph P}(X)\]
for all $P\leq S$ and $\ph\in \cal F(P,S)$; which proves \ref{itemPhiBurnsideEq}$\Rightarrow$\ref{itemPhiStable}.

Assume \ref{itemPhiStable}. Let $P\leq S$ and $\ph\in \cal F(P,S)$. By assumption, we have $\Phi_{\ph R}(X) = \Phi_{R}(X)$ for all $R\leq P$, hence
\[\Phi^P_{R}(\prescript{}{P,\ph} X) =  \Phi_{\ph R}(X) = \Phi_{R}(X) = \Phi^P_{R}(\prescript{}{P,incl} X).\]
Since $\Phi^P$ is injective, we get $\prescript{}{P,\ph} X = \prescript{}{P,incl} X$; so \ref{itemPhiStable}$\Rightarrow$\ref{itemPhiBurnsideEq}.

Finally, we have \ref{itemPhiStable}$\Leftrightarrow$\ref{itemFConjStable} because $Q$ is $\cal F$-conjugate to $P$ exactly when $Q$ is the image of a map $\ph\in \cal F(P,S)$ in the fusion system.
\end{proof}

\begin{dfn}\label{dfnMonoidOfStableSets}
We let $A_+(\cal F)\subseteq A_+(S)$ be the set of all the $\cal F$-stable sets, and by property \ref{itemFConjStable} the sums and products of stable elements are still stable, so $A_+(\cal F)$ is a subsemiring of $A_+(S)$.
\end{dfn}

Suppose that $\cal F=\cal F_S(G)$ is the fusion system for a group with $S\in Syl_p(G)$.
Let $X\in A_+(G)$ be a $G$-set, and let $\prescript{}S X$ be the same set with the action restricted to the Sylow $p$-subgroup $S$. If we let $P\leq S$ and $c_g\in \Hom_{\cal F_S(G)}(P,S)$ be given; then $x\mapsto gx$ is an isomorphism of $P$-sets $\prescript{}{P,incl} X \cong \prescript{}{P,c_g} X$. The restriction $\prescript{}{S,incl} X$ is therefore $G$-stable.

Restricting the group action from $G$ to $S$ therefore defines a homomorphism of semirings $A_+(G) \to A_+(\cal F_S(G))$, but as the following example shows, this map need not be injective nor surjective.

\begin{ex}\label{exNonUnique}
The symmetric group $\Sigma_5$ on $5$ letters has Sylow $2$-subgroups isomorphic to the dihedral group $\D_8$ of order $8$. We then consider $\D_8$ as embedding in $\Sigma_5$ as one of the Sylow $2$-subgroups.
Let $H,K$ be respectively Sylow $3$- and $5$-subgroups of $\Sigma_5$.

The transitive $\Sigma_5$-set $[\Sigma_5/H]$ contains $40$ elements and all the stabilizers have odd order (they are conjugate to $H$). When we restrict the action to $\D_8$, the stabilizers therefore become trivial so the $\D_8$-action is free, hence $[\Sigma_5/H]$ restricts to the $\D_8$-set $5\cdot [\D_8/1]$, that is $5$ disjoint copies of the free orbit $[\D_8/1]$.
Similarly, the transitive $\Sigma_5$-set $[\Sigma_5/K]$ restricts to $3\cdot[\D_8/1]$.

These two restrictions of $\Sigma_5$-sets are not linearly independent as $\D_8$-sets -- the $\Sigma_5$-sets $3\cdot [\Sigma_5/H]$ and $5\cdot [\Sigma_5/K]$ both restrict to $15\cdot [\D_8/1]$. If the restrictions of $\Sigma_5$-sets were to form a free abelian monoid, then the set $[\D_8/1]$ would have to be the restriction of an $\Sigma_5$-set as well; and since $[\D_8/1]$ is irreducible as a $\D_8$-set, it would have to be the restriction of an irreducible (hence transitive) $\Sigma_5$-set. However $\Sigma_5$ has no subgroup of index $8$, hence there is no transitive $\Sigma_5$ with $8$ elements.

This shows that the restrictions of $\Sigma_5$-sets to $\D_8$ do not form a free abelian monoid, and we also see that $[\D_8/1]$ is an example of an $\cal F_{\D_8}(\Sigma_5)$-stable set ($\Phi_1([\D_8/1])=8$ and $\Phi_Q([\D_8/1])=0$ for $1\neq Q\leq \D_8$) which cannot be given the structure of an $\Sigma_5$-set.
\end{ex}

To define the Burnside ring of a fusion system $\cal F$, we have two possibilities: We can consider the semiring of all the $\cal F$-stable $S$-sets and take the Grothendieck group of this. Alternatively, we can first take the Grothendieck group for all $S$-sets to get the Burnside ring of $S$, and then afterwards we consider the subring herein consisting of all the $\cal F$-stable elements. The following proposition implies that the two definitions coincide for saturated fusion systems.
\begin{prop}\label{propGrothendieckFusion}
Let $\cal F$ be a fusion system over a $p$-group $S$, and consider
the subsemiring $A_+(\cal F)$ of $\cal F$-stable $S$-sets in the semiring $A_+(S)$ of finite $S$-sets.

This inclusion induces a ring homomorphism from the Grothendieck group of $A_+(\cal F)$ to the Burnside ring $A(S)$, which is injective.

If $\cal F$ is saturated, then the image of the homomorphism is the subring of $A(S)$ consisting of the $\cal F$-stable elements.
\end{prop}

\begin{proof}
Let $Gr$ be the Grothendieck group of $A_+(\cal F)$, and let $I\colon Gr\to A(S)$ be the induced group homomorphism coming from the inclusion $i\colon A_+(\cal F)\into A_+(S)$.

An element of $Gr$ is a formal difference $X-Y$ where $X$ and $Y$ are $\cal F$-stable sets. Assume now that $X-Y$ lies in $\ker I$. This means that $i(X)-i(Y)=0$ in $A(S)$; and since $A_+(S)$ is a free commutative monoid, we conclude that $i(X)=i(Y)$ as $S$-sets. But $i$ is just the inclusion map, so we must have $X=Y$ in $A_+(\cal F)$ as well, and $X-Y=0$ in $Gr$. Hence $I\colon Gr\to A(S)$ is injective.

It is clear that the difference of two $\cal F$-stable sets is still $\cal F$-stable, so $\im I$ lies in the subring of $\cal F$-stable elements. If $\cal F$ is saturated, then the converse holds, and all $\cal F$-stable elements of $A(S)$ can be written as a difference of $\cal F$-stable sets; however the proof of this must be postponed to Corollary \ref{corBurnsideBasis} below.
\end{proof}

\begin{dfn}
Let $\cal F$ be saturated. We define \emph{the Burnside ring of $\cal F$}, denoted $A(\cal F)$, to be the subring consisting of the $\cal F$-stable elements in $A(S)$.

Once we have proven Corollary \ref{corBurnsideBasis}, we will know that $A(\cal F)$ is also the Grothendieck group of the semiring $A_+(\cal F)$ of $\cal F$-stable sets.
\end{dfn}

\subsection{Proving Theorems \ref{thmMonoidBasis} and \ref{thmYoshidaFusion}}\label{secProofOfTheorem}
The proof of Theorem \ref{thmMonoidBasis} falls into several parts: We begin by constructing some $\cal F$-stable sets $\alpha_P$ satisfying certain properties -- this is the content of \ref{lemInduceFromAStoAF}-\ref{propIrreducibleSets}. We construct one $\alpha_P$ per $\cal F$-conjugacy class of subgroups, and these are the $\cal F$-stable sets which we will later show are the irreducible stable sets. A special case of the construction was originally used by Broto, Levi and Oliver in \cite{BrotoLeviOliver}*{Proposition 5.5} to show that every saturated fusion system has a characteristic biset.

In \ref{corLinIndependent}-\ref{corBurnsideBasis} we then proceed to show that the constructed $\alpha_P$'s are linearly independent, and that they generate the Burnside ring $A(\cal F)$. When proving that the $\alpha_P$'s generate $A(\cal F)$, the same proof also establishes Theorem \ref{thmYoshidaFusion}.

Finally, we use the fact that the $\alpha_P$'s form a basis for the Burnside ring, to argue that they form an additive basis already for the semiring $A_+(\cal F)$, completing the proof of Theorem \ref{thmMonoidBasis} itself.

\medskip
As mentioned, we first construct an $\cal F$-stable set $\alpha_P$ for each $\cal F$-conjugation class of subgroups. The idea when constructing $\alpha_P$ is that we start with the single orbit $[S/P]$ which we then stabilize: We run through the subgroups $Q\leq S$ in decreasing order and add orbits to the constructed $S$-set such that it becomes $\cal F$-stable at the conjugacy class of $Q$ in $\cal F$.
The stabilization procedure is handled in the following technical Lemma \ref{lemInduceFromAStoAF}, which is then applied in Proposition \ref{propIrreducibleSets} to construct the $\alpha_P$'s.

Recall that $c_P(X)$ denotes the number of $(S/P)$-orbits in $X$, and $\Phi_P(X)$ denotes the number of $P$-fixed points.
\begin{lem}
\label{lemInduceFromAStoAF}
Let $\cal F$ be a saturated fusion system over a $p$-group $S$, and let $\cal H$ be a collection of subgroups of $S$ such that $\cal H$ is closed under taking $\cal F$-subconjugates, i.e. if $P\in \cal H$, then $Q\in \cal H$ for all $Q\lesssim_{\cal F} P$.

Assume that $X\in A_+(S)$ is an $S$-set satisfying $\Phi_P(X) = \Phi_{P'}(X)$ for all pairs $P\sim_{\cal F}\nobreak P'$, with $P,P'\not\in \cal H$.
Assume furthermore that $c_{P}(X) = 0$ for all $P\in \cal H$.

Then there exists an $\cal F$-stable set $X'\in A_+(\cal F)\subseteq A_+(S)$ satisfying $\Phi_P(X') = \Phi_P(X)$ and $c_{P}(X') = c_{P}(X)$ for all $P\not\in \cal H$; and also satisfying $c_{P}(X') = c_{P}(X)$ for all $P\leq S$ which are fully normalized in $\cal F$.
In particular, for a $P\in \cal H$ which is fully normalized, we have $c_{P}(X')=0$.
\end{lem}

\begin{proof}
We proceed by induction on the size of the collection $\cal H$.
If $\cal H=\Ø$, then $X$ is $\cal F$-stable by assumption, so $X':= X$ works.

Assume that $\cal H\neq \Ø$, and let $P\in \cal H$ be maximal under $\cal F$-subconjugation as well as fully normalized.

Let $P'\sim_{\cal F} P$. Then there is a homomorphism $\ph\in {\cal F}(N_S P', N_S P)$ with $\ph(P')=P$ by Lemma \ref{lemNormalizerMap} since $\cal F$ is saturated.
The restriction of $S$-actions to $\ph(N_S P')$ gives a ring~homo\-morphism $A(S) \to A(\ph(N_S P'))$ that preserves the fixed-point homomorphisms $\Phi_{Q}$ for $Q\leq \ph(N_S P')\leq N_S P$.

If we consider the $S$-set $X$ with the action restricted to $\ph(N_S P')$, we can apply Corollary \ref{corYoshidaGroup} for the normal subgroup $P=\ph (P')\unlhd \ph(N_S P')$ to get
\[\sum_{\bar s\in \ph(N_S P')/ P} \Phi_{\gen s P}(X) \equiv 0 \pmod{\abs{\ph(N_S P')/P}}.\]
Similarly, we have $P'\unlhd N_S P'$, with which Corollary \ref{corYoshidaGroup} gives us
\[\sum_{\bar s\in N_S P'/P'} \Phi_{\gen s P'}(X) \equiv 0 \pmod{\abs{N_S P'/P'}}.\]
Since $P$ is maximal in $\cal H$, we have by assumption $\Phi_{Q}(X) = \Phi_{Q'}(X)$ for all $Q\sim_{\cal F} Q'$ where $P$ is $\cal F$-conjugate to a \emph{proper} subgroup of $Q$. Specifically, we have
\[\Phi_{\gen {\ph (s)}P}(X) = \Phi_{\ph(\gen sP')}(X) = \Phi_{\gen sP'}(X)\] for all $s\in N_S P'$ with $s\not\in P'$. It then follows that
\begin{align*}
\Phi_{P}(X)-\Phi_{P'}(X)&= \sum_{\bar s\in \ph(N_S P')/P} \Phi_{\gen s P}(X) - \sum_{\bar s\in N_S P' / P'} \Phi_{\gen s P'}(X)
\\ &\equiv 0-0 \pmod{\abs{W_SP'}}.
\end{align*}
We can therefore define $\lambda_{P'} := (\Phi_{P}(X)-\Phi_{P'}(X)) / \abs{W_SP'}\in \Z$.

Using the $\lambda_{P'}$ as coefficients, we construct a new $S$-set
\[\tilde X:=X +\!\! \sum_{[P']_S\subseteq [P]_{\cal F}} \!\!\lambda_{P'}\cdot[S/P'] \quad\in A(S).\]
Here $[P]_{\cal F}$ is the collection of subgroups that are $\cal F$-conjugate to $P$. The sum is then taken over one representative from each $S$-conjugacy class contained in $[P]_{\cal F}$.

A priori, the $\lambda_{P'}$ might be negative, and as a result $\tilde X$ might not be an $S$-set. In the original construction of \cite{BrotoLeviOliver}, this problem is circumvented by adding copies of
\[\sum_{[P']_S \subseteq [P]_{\cal F}} \frac{\abs{N_S P}}{\abs{N_S P'}}\cdot [S/P']\]
until all the coefficients are non-negative.

It will however be shown in Lemma \ref{lemMoreFixpoints} below, that under the assumption that $c_{P'}(X)=0$ for $P'\sim_{\cal F} P$, then $\lambda_{P'}$ is always non-negative, and $\lambda_{P'}=0$ if $P'$ is fully normalized. Hence $\tilde X$ is already an $S$-set without further adjustments.

We clearly have $c_{Q}(\tilde X) = c_{Q}(X)$ for all $Q\not\sim_{\cal F} P$, in particular for all $Q\not\in \cal H$. Furthermore, if $P'\sim_{\cal F} P$ is fully normalized, then $c_{P'}(\tilde X)=c_{P'}(X)+\lambda_{P'}=c_{P'}(X)$.

Because $\Phi_{Q}([S/P']) =0$ unless $Q\lesssim_S P'$, we see that $\Phi_{Q}(\tilde X) = \Phi_{Q}(X)$ for every $Q\not\in \cal H$. Secondly, we calculate $\Phi_{P'}(\tilde X)$ for each $P'\sim_{\cal F} P$:
\begin{align*}
\Phi_{P'}(\tilde X) &= \Phi_{P'}(X) + \sum_{[\tilde P]_S\subseteq [P]_{\cal F}} \lambda_{\tilde P}\cdot\Phi_{P'}([S/\tilde P])
\\ &= \Phi_{P'}(X) + \lambda_{P'}\cdot\Phi_{P'}([S/P']) = \Phi_{P'}(X) + \lambda_{P'}\abs{W_S P'}
\\ &= \Phi_{P}(X);
\end{align*}
which is independent of the choice of $P'\sim_{\cal F} P$.

We define $\cal H':= \cal H \setminus [P]_{\cal F}$ as $\cal H$ with the $\cal F$-conjugates of $P$ removed. Because $P$ is maximal in $\cal H$, the subcollection $\cal H'$ again contains all $\cal F$-subconjugates of any $H\in \cal H'$.

By induction we can apply Lemma \ref{lemInduceFromAStoAF} to $\tilde X$ and the smaller collection $\cal H'$. We get an $X'\in A_+(\cal F)$ with $\Phi_{Q}(X') = \Phi_{Q}(\tilde X)$ and $c_{Q}(X') = c_{Q}(\tilde X)$ for all $Q\not\in \cal H'$; and such that $c_{Q}(X')=0$ if $Q\in \cal H'$ is fully normalized.

It follows that $\Phi_{Q}(X') = \Phi_{Q}(\tilde X)=\Phi_{Q}(X)$ and $c_{Q}(X')=c_{Q}(\tilde X) = c_{Q}(X)$ for all $Q\not\in \cal H$, and we also have $c_{Q}(X')=0$ if $Q\in \cal H$ is fully normalized.
\end{proof}

\begin{lem}\label{lemMoreFixpoints}
Let $\cal F$ be a saturated fusion system over a $p$-group $S$, and let $P\leq S$ be a fully normalized subgroup.

Suppose that $X$ is an $S$-set with $c_{P'}(X)=0$ for all $P'\sim_{\cal F} P$, and satisfying that $X$ is already $\cal F$-stable for subgroups larger than $P$, i.e.
$\abs[\big]{X^R}=\abs[\big]{X^{R'}}$ for all $R\sim_{\cal F} R'$ where $P$ is $\cal F$-conjugate to a proper subgroup of $R$.

Then $\abs[\big]{X^P} \geq \abs[\big]{X^{P'}}$ for all $P'\sim_{\cal F}P$.
\end{lem}

\begin{proof}
Let $Q\sim_{\cal F} P$ be given. Because $P$ is fully normalized, there exists by Lemma \ref{lemNormalizerMap} a homomorphism $\ph\colon N_S Q \into N_S P$ in $\cal F$, with $\ph(Q)=P$.

Let $A_1,\dotsc,A_k$ be the subgroups of $N_S Q$ that strictly contain $Q$, i.e. $Q< A_i\leq N_S Q$. We put $B_i:= \ph(A_i)$, and thus also have $P< B_i \leq N_S P$. We let $C_1,\dotsc,C_\ell$ be the subgroups of $N_SP$ strictly containing $P$ which are not the image (under $\ph$) of some $A_i$. Hence $B_1,\dotsc,B_k,C_1,\dotsc,C_\ell$ are all the different subgroups of $N_SP$ strictly containing $P$.
We denote the set $\{1,\dotsc,k\}$ of indices by $I$, and also $J:=\{1,\dotsc,\ell\}$.

Because $c_{Q}(X)=c_{P}(X)=0$ by assumption, no orbit of $X$ is isomorphic to $S/Q$, hence no element in $X^Q$ has $Q$ as a stabilizer. Let $x\in X^Q$ be any element, and let $K> Q$ be the stabilizer of $x$; so $x\in X^K \subseteq X^Q$. Since $K$ is a $p$-group, there is some intermediate group $L$ with $Q\lhd L \leq K$; hence $x\in X^L$ for some $Q<L\leq N_SQ$. We conclude that
\[X^Q=\bigcup_{i\in I} X^{A_i}.\]
With similar reasoning we also get
\[X^P=\bigcup_{i\in I} X^{B_i}\cup\bigcup_{j\in J} X^{C_j}.\]
The proof is then completed by showing
\[\abs*{X^P} = \abs[\Big]{\bigcup_{i\in I} X^{B_i}\cup\bigcup_{j\in J} X^{C_j}} \geq \abs[\Big]{\bigcup_{i\in I} X^{B_i}} \overset{(*)}= \abs[\Big]{\bigcup_{i\in I} X^{A_i}} = \abs*{X^Q}.\]
We only need to prove the equality $(*)$.

Showing $(*)$ has only to do with fixed points for the subgroups $A_i$ and $B_i$; and because $B_i=\ph(A_i) \sim_{\cal F} A_i$ are subgroups that strictly contain $P$ and $Q$ respectively, we have $\abs*{X^{B_i}}=\abs*{X^{A_i}}$ by assumption.

To get $(*)$ for the unions $\cup A_i$ and $\cup B_i$ we then have to apply the inclusion-exclusion principle:
\[\abs[\Big]{\bigcup_{i\in I} X^{B_i}} = \sum_{\Ø\neq \Lambda \subseteq I} (-1)^{\abs* \Lambda +1} \abs[\Big]{\bigcap_{i\in \Lambda} X^{B_i}} = \sum_{\Ø\neq \Lambda \subseteq I} (-1)^{\abs* \Lambda +1} \abs*{X^{\gen{B_i}_{i\in \Lambda}}}.\]
Here $\gen{B_i}_{i\in \Lambda}\leq N_SP$ is the subgroup generated by the elements of $B_i$'s with $i\in \Lambda\subseteq I$. Recalling that $B_i=\ph(A_i)$ by definition, we have $\gen{B_i}_{i\in \Lambda} = \gen{\ph(A_i)}_{i\in \Lambda} = \ph(\gen{A_i}_{i\in \Lambda})$, and consequently
\[\sum_{\Ø\neq \Lambda \subseteq I} (-1)^{\abs* \Lambda +1} \abs*{X^{\gen{B_i}_{i\in \Lambda}}} = \sum_{\Ø\neq \Lambda \subseteq I} (-1)^{\abs* \Lambda +1} \abs*{X^{\ph(\gen{A_i}_{i\in \Lambda})}}.\]
Because $Q<A_i\leq N_S Q$, we also have $Q<\gen{A_i}_{i\in \Lambda}\leq N_S Q$, by assumption we therefore get $\abs*{X^{\ph(\gen{A_i}_{i\in \Lambda})}} = \abs*{X^{\gen{A_i}_{i\in \Lambda}}}$ for all $\Ø\neq \Lambda \subseteq I$. It then follows that
\[\sum_{\Ø\neq \Lambda \subseteq I} (-1)^{\abs* \Lambda +1} \abs*{X^{\ph(\gen{A_i}_{i\in \Lambda})}} = \sum_{\Ø\neq \Lambda \subseteq I} (-1)^{\abs* \Lambda +1} \abs*{X^{\gen{A_i}_{i\in \Lambda}}} = \dotsb = \abs[\Big]{\bigcup_{i\in I} X^{A_i}},\]
where we use the inclusion-exclusion principle in reverse. We have thus shown the equality $\abs*{\bigcup_{i\in I} X^{B_i}} = \abs*{\bigcup_{i\in I} X^{A_i}}$ as required.
\end{proof}

Applying the technical Lemma \ref{lemInduceFromAStoAF}, we can now construct the irreducible $\cal F$-stable sets $\alpha_P$ for $P \leq S$ as described in the following proposition. That the $\alpha_P$'s are in fact irreducible, or even that they are unique, will not be shown until the proof of Theorem \ref{thmMonoidBasis} itself.
\begin{prop}\label{propIrreducibleSets} Let $\cal F$ be a saturated fusion system over a $p$-group $S$.

For each $\cal F$-conjugacy class $[P]_{\cal F}\in \ccset{\cal F}$ of subgroups, there is an $\cal F$-stable set $\alpha_P\in A_+(\cal F)$ such that
\begin{enumerate}\renewcommand{\theenumi}{(\roman{enumi})}\renewcommand{\labelenumi}{\theenumi}
\item\label{itemSmaller} $\Phi_{Q}(\alpha_P) = 0$ unless $Q$ is $\cal F$-subconjugate to $P$.
\item\label{itemWeyl} $c_{P'}(\alpha_P)=1$ and $\Phi_{P'}(\alpha_P) = \abs{W_S P'}$ when $P'$ is fully normalized and $P'\sim_{\cal F} P$.
\item\label{itemNoOtherFullyNormalized} $c_{Q}(\alpha_P) = 0$ when $Q$ is fully normalized and $Q\not\sim_{\cal F} P$.
\end{enumerate}
\end{prop}

\begin{proof}
Let $P\leq S$ be fully $\cal F$-normalized. We let $X\in A_+(S)$ be the $S$-set
\[X:= \sum_{[P']_S \subseteq [P]_{\cal F}} \frac{\abs{N_S P}}{\abs{N_S P'}}\cdot [S/P']\quad \in A_+(S).\]
$X$ then satisfies that $\Phi_{Q}(X)=0$ unless $Q\lesssim_{S} P'$ for some $P'\sim_{\cal F} P$, in which case we have $Q\lesssim_{\cal F} P$. For all $P',P''\in [P]_{\cal F}$ we have $\Phi_{P''}([S/P']) = 0$ unless $P''\sim_{S} P'$; and consequently
\[\Phi_{P'}(X) = \frac{\abs{N_S P}}{\abs{N_S P'}}\cdot \Phi_{P'}([S/P'])
= \frac{\abs{N_S P}}{\abs{N_S P'}}\cdot \abs{W_S P'} = \abs{W_S P}\]
which doesn't depend on $P'\sim_{\cal F} P$.

Let $\cal H$ be the collection of all $Q$ which are $\cal F$-conjugate to a \emph{proper} subgroup of $P$, then $\Phi_{Q}(X) = \Phi_{Q'}(X)$ for all pairs $Q\sim_{\cal F} Q'$ not in $\cal H$. Using Lemma \ref{lemInduceFromAStoAF} we get some $\alpha_P\in A_+(\cal F)$ with the required properties.
\end{proof}

Properties \ref{itemWeyl} and \ref{itemNoOtherFullyNormalized} make it really simple to decompose a linear combination $X$ of the $\alpha_P$'s. The coefficient of $\alpha_P$ in $X$ is just the number of $[S/P]$-orbits in $X$ as an $S$-set - when $P$ is fully normalized. This is immediate since $\alpha_P$ contains exactly one copy of $[S/P]$, and no other $\alpha_Q$ contains $[S/P]$.

\enlargethispage{2\baselineskip}
In particular we have:
\begin{cor}\label{corLinIndependent}
The $\alpha_P$'s in Proposition \ref{propIrreducibleSets} are linearly independent.
\end{cor}

In order to prove that the $\alpha_P$'s generate all $\cal F$-stable sets, we will first show that the $\alpha_P$'s generate all the $\cal F$-stable elements in the Burnside ring. As a tool for proving this, we define a ghost ring for the Burnside ring $A(\cal F)$; and as consequence of how the proof proceeds, we end up showing an analogue of Proposition \ref{propYoshidaGroup} for saturated fusion systems, describing how the Burnside ring $A(\cal F)$ lies embedded in the ghost ring -- this is the content of Theorem \ref{thmYoshidaFusion}.

\begin{dfn}\label{dfnPhiF}
We defined the ghost ring $\free S$ for the Burnside ring of a group as the product ring $\prod_{[P]_S \in \ccset{S}} \Z$ where the coordinates correspond to the $S$-conjugacy classes of subgroups.
For the ring $A(\cal F)$, we now similarly define \emph{the ghost ring $\free{\cal F}$} as a product ring $\prod_{[P]_{\cal F} \in \ccset{\cal F}} \Z$ with coordinates corresponding to the $\cal F$-conjugacy classes of subgroups.

The surjection of indexing sets $\ccset{S} \to \ccset{\cal F}$ which sends an $S$-conjugacy class $[P]_S$ to its $\cal F$-conjugacy class $[P]_{\cal F}$, induces a homomorphism $\free{\cal F}\into\free S$ that embeds $\free{\cal F}$ as the subring of vectors which are constant on each $\cal F$-conjugacy class.

Since $A(\cal F)$ is the subring of $\cal F$-stable elements in $A(S)$, we can restrict the mark homomorphism $\Phi^S\colon A(S)\to \free S$ to the subring $A(\cal F)$ and get an injective ring homomorphism $\Phi^{\cal F}\colon A(\cal F) \to \free{\cal F}$ -- which is the \emph{homomorphism of marks} for $A(\cal F)$.

To model the cokernel of $\Phi^{\cal F}$ we define $Obs(\cal F)$ as
\[Obs(\cal F) := \prod_{\substack{[P]\in\ccset{\cal F}\\ P\text{ f.n.}}} (\Z / \abs{W_S P}\Z),\]
where 'f.n.' is short for 'fully normalized', so we take fully normalized representatives of the conjugacy classes in $\cal F$.
\end{dfn}

\setcounter{mainthm}1
\begin{mainthm}\label{thmYoshidaFusion}\label{dfnPsiF}
Let $\cal F$ be a saturated fusion system over a $p$-group $S$, and let $A(\cal F)$ be the Burnside ring of $\cal F$, i.e. the subring consisting of the $\cal F$-stable elements in the Burnside ring of $S$.
We then have a short-exact sequence
\[
0 \to A(\cal F) \xto{\Phi}  \free {\cal F} \xto{\Psi} Obs(\cal F) \to 0.
\]
where $\Phi=\Phi^{\cal F}$ is the homomorphism of marks, and $\Psi=\Psi^{\cal F}\colon \free{\cal F} \to Obs(\cal F)$ is a group homomorphism given by the $[P]$-coordinate functions
\[\Psi_{P}(\xi)  := \sum_{\bar s\in W_S P} \xi_{\gen s P} \pmod {\abs{W_S P}}\]
when $P\leq S$ is a fully normalized representative of the conjugacy class $[P]$ in $\cal F$.
Here $\Psi_P=\Psi_{P'}$ if $P\sim_{\cal F} P'$ are both fully normalized.
\end{mainthm}

\begin{proof}
We choose some total order of the conjugacy classes $[P],[Q]\in \ccset {\cal F}$ such that $\abs{P} > \abs {Q}$ implies $[P]< [Q]$, i.e. we take the subgroups in decreasing order. In holds in particular that $Q\lesssim_{\cal F} P$ implies $[P] \leq [Q]$.

With respect to the ordering above, the group homomorphism $\Psi$ is given by a lower triangular matrix with $1$'s in the diagonal, hence $\Psi$ is surjective. The mark homomorphism $\Phi=\Phi^{\cal F}$ is the restriction of the injective ring homomorphism $\Phi^S\colon A(S) \to \free{S}$, so $\Phi$ is injective.

We know from the group case, Proposition \ref{propYoshidaGroup}, that $\Psi^S\circ \Phi^S=0$. By construction we have $(\Psi)_P = (\Psi^S)_P$ for the coordinate functions when $P$ is fully normalized; and $\Phi$ is the restriction of $\Phi^S$. We conclude that $\Psi\circ \Phi=0$ as well.
It remains to be shown that $\im\Phi$ is actually all of $\ker \Psi$.

Consider the subgroup $H := \Span\{\alpha_P \mid [P]\in \ccset{\cal F}\}$ spanned by the $\alpha_P$'s in $A(\cal F)$, and consider also the restriction $\Phi|_H$ of the mark homomorphism $\Phi\colon A(\cal F) \to \free{\cal F}$.

$\Phi|_H$ is described by a square matrix $M$ in terms of the ordered bases of $H=\Span\{\text{$\alpha_P$'s}\}$ and $\free {\cal F}$.
Because $ M_{[Q],[P]} :=\Phi_{Q}(\alpha_P)$ is zero unless $P\sim_{\cal F} Q$ or $\abs P > \abs Q$, we conclude that $M$ is a lower triangular matrix. The diagonal entries of $M$ are
\[M_{[P],[P]} = \Phi_{P}(\alpha_P) = \abs{W_S P},\]
when $P$ is fully normalized.

All the diagonal entries are non-zero, so the cokernel of $\Phi|_H$ is finite of order
\[\abs{\coker \Phi|_H} = \prod_{[P]\in \ccset {\cal F}} M_{[P],[P]}= \prod_{\substack{[P]\in\ccset{\cal F}\\ P\text{ f.n.}}} \abs{W_S P}.\]
Since $\Phi|_H$ is a restriction of $\Phi$, it follows that $\abs{\coker \Phi} \leq \abs{\coker \Phi|_H}$.
At the same time, $\Psi\circ \Phi=0$ implies that $\abs{\coker \Phi} \geq \abs{Obs(\cal F)}$.

We do however have
\[\abs{Obs(\cal F)} = \prod_{\substack{[P]\in\ccset{\cal F}\\ P\text{ f.n.}}} \abs{W_S P} = \abs{\coker \Phi|_H}.\]
The only possibility is that $\ker \Psi =\im\Phi = \im \Phi|_H$, completing the proof of Theorem \ref{thmYoshidaFusion}.
\end{proof}

From the last equality $\im\Phi = \im\Phi|_H$, and the fact that $\Phi$ is injective, it also follows that $A(\cal F)= H$ so the $\alpha_P$'s span all of $A(\cal F)$. Combining this with Corollary \ref{corLinIndependent} we get:
\begin{cor}\label{corBurnsideBasis}
The $\alpha_P$'s form an additive basis for the Burnside ring $A(\cal F)$.
\end{cor}

The corollary tells us that any element $X\in A(\cal F)$ can be written uniquely as an integral linear combination of the $\alpha_P$'s. In particular, any $\cal F$-stable set can be written as a linear combination of $\alpha_P$'s, and if the coefficients are all non-negative, then we have a linear combination in $A_+(\cal F)$.

\setcounter{mainthm}{0}
\begin{mainthm}\label{thmMonoidBasis}
Let $\cal F$ be a saturated fusion system over a $p$-group $S$.
The sets $\alpha_P$ from Proposition \ref{propIrreducibleSets} are all the irreducible $\cal F$-stable sets, and every $\cal F$-stable set splits uniquely, up to $S$-isomorphism, as a disjoint union of the $\alpha_P$'s.

Hence the semiring $A_+(\cal F)$ of $\cal F$-stable sets is additively a free commutative monoid with rank equal to the number of conjugacy classes of subgroups in $\cal F$.
\end{mainthm}

\begin{proof}
Let $\alpha_P\in A_+(\cal F)$ for each conjugacy class $[P]\in \ccset{\cal F}$ be given as in Proposition \ref{propIrreducibleSets}. Let $X\in A_+(\cal F)$ be any $\cal F$-stable $S$-set.

Since the $\alpha_P$'s form a basis for $A(\cal F)$ by Corollary \ref{corBurnsideBasis}, we can write $X$ uniquely as
\[X= \sum_{[P]\in \ccset{\cal F}} \lambda_{P}\cdot \alpha_P\]
with $\lambda_{P}\in \Z$.

\enlargethispage{2\baselineskip}
Suppose that $P$ is fully normalized, then $c_{P}(\alpha_Q)=1$ if $P\sim_{\cal F} Q$, and $c_{P}(\alpha_Q)=0$ otherwise. As a consequence of this, we have
\[c_{P}(X) = \sum_{[Q]\in \ccset{\cal F}} \lambda_{Q}\cdot c_{P}( \alpha_Q) = \lambda_{P}\]
whenever $P$ is fully normalized.

Because $X$ is an $S$-set, we see that $\lambda_{P} = c_{P}(X)\geq 0$. Hence the linear combination $X= \sum_{[P]\in \ccset{\cal F}} \lambda_{P}\cdot \alpha_P$ has nonnegative coefficients, i.e. it is a linear combination in the semiring $A_+(\cal F)$.

As a special case, if we have another element $\alpha'_P$ in $A(\cal F)$ satisfying the properties of Proposition \ref{propIrreducibleSets}, then the fact that $\lambda_{Q} = c_{Q}(\alpha'_P)$ for all fully normalized $Q\leq S$, shows that $\lambda_{P}=1$ and $\lambda_{Q}=0$ for $Q\not\sim_{\cal F} P$. Thus the linear combination above simplifies to $\alpha'_P = \alpha_P$. Hence the $\alpha_P$'s are uniquely determined by the properties of Proposition \ref{propIrreducibleSets}.
\end{proof}

\appendix
\section{The monoid of complex representations}\label{secRepresentations}
For a saturated fusion system $\cal F$ over $S$, it makes sense to talk about $\cal F$-stability of $S$-representations instead of $S$-sets. In this appendix we show that the analogue of Theorems \ref{thmMonoidBasis} and \ref{thmMonoidBasisIntroGroup} fails for representations in general by giving an example where the abelian monoid of $\cal F$-stable complex representations is not free.

For a finite dimensional complex representation $\rho\colon S\to GL_n(\C)$ of $S$, we can restrict $\rho$ along any fusion map $\ph\in \cal F(P,S)$ to form a representation $\prescript{}{P,\ph}\rho:=\rho\circ \ph$ of the subgroup $P\leq S$. Just as for finite $S$-sets, we compare each $\prescript{}{P,\ph}\rho$ to the usual restriction $\prescript{}{P,incl}\rho$ and say that $\rho$ is $\cal F$-stable if
\begin{equation*}
\parbox[c]{.9\textwidth}{\emph{$\prescript{}{P,\ph}\rho$ is isomorphic to $\prescript{}{P,incl}X$ as representations of $P$, for all $P\leq S$ and homomorphisms $\ph\colon P\to S$ in $\cal F$.}}
\end{equation*}
The isomorphism classes of $\cal F$-stable complex $S$-representations form an abelian monoid $R_+(\cal F)$ with direct sum of representations as the addition. As we know, the isomorphism class of any complex character is determined completely by the associated character. Our first order of business is therefore to determine which characters belong to $\cal F$-stable representations. We say that a character $\chi\colon S\to \C$ is $\cal F$-stable if it satisfies $\chi(s)=\chi(\ph(s))$ for all elements $s\in S$ and maps $\ph\in \cal F(\gen{s},S)$, that is $\chi$ should be constant on each conjugacy class in $\cal F$ of elements in $S$.
\begin{lem}\label{lemFstablechar}
Let $\rho\colon S\to GL_n(\C)$ be a representation, and let $\chi\colon S\to \C$ be the associated character.
Then $\rho$ is $\cal F$-stable if and only if $\chi$ is $\cal F$-stable.
\end{lem}

\begin{proof}
Consider a subgroup $P\leq S$ and a map $\ph\in\cal F(P,S)$, then the character associated to the restriction $\prescript{}{P,\ph}\rho = \rho \circ \ph$ is equal to $\chi\circ \ph$. The representation $\prescript{}{P,\ph}\rho$ is isomorphic to $\prescript{}{P,incl}\rho$ precisely when they have the same character on $P$, that is whenever $\chi\circ \ph = \chi|_P$, which on elements becomes $\chi(\ph(s)) = \chi(s)$ for all $s\in P$.

We now immediately conclude that $\rho$ is $\cal F$-stable if and only if $\chi(\ph(s)) = \chi(s)$ for all $s\in P$, $P\leq S$ and $\ph\in \cal F(P,S)$. By restricting each $\ph$ to the cyclic subgroup $\gen{s}\leq P$, it is enough to check that $\chi(\ph(s)) = \chi(s)$ for all $s\in S$ and $\ph\in \cal F(\gen{s},S)$, i.e. that $\chi$ is $\cal F$-stable.
\end{proof}

Using lemma \ref{lemFstablechar} to characterize the $\cal F$-stable representations, we will now study the example below and see that $R_+(\cal F)$ is not a free abelian monoid for this particular choice of $\cal F$.
\begin{ex}\label{exCounterexampleRepr}
We consider the saturated fusion system $\cal F:=\cal F_{\SD_{16}}(\PGL_3(\F_3))$ induced by the projective general linear group $\PGL_3(\F_3)$ on the semidihedral group of order $16$, i.e. the group $\SD_{16}=\gen{D,S \mid D^8=S^2=1, SDS^{-1} = D^3}$. One possible inclusion of $\SD_{16}$ inside $\PGL_3(\F_3)$ has as matrix representatives:
\[\begin{pmatrix}1 & 0 & 0 \\ 0 & 1 & 1 \\ 2 & 1 & 0\end{pmatrix}\text{ represents $D$, and }\begin{pmatrix}1 & 0 & 0 \\ 0 & 1 & 2\\ 0 & 0 & 2\end{pmatrix}\text{ represents $S$.}\]
The group $\SD_{16}$ has $7$ conjugacy classes of elements, and $7$ irreducible characters. These are all listed in the character table below:
\begin{equation*}
\begin{array}{c|ccccccc}
    & 1 & D,D^3 & D^5,D^7 & D^2,D^6 & D^4 & S,D^2S,D^4S,D^6S & DS, D^3S, D^5S, D^7S \\\hline
 \textbf{1} & 1 & 1 &1 &1 &1 &1 &1 \\
 \chi_a & 1 & -1 & -1 & 1 & 1 & -1 & 1 \\
 \chi_b & 1 & -1 & -1 & 1 & 1 & 1 & -1 \\
 \chi_{ab} & 1 & 1 & 1 & 1 & 1 & -1 & -1 \\
 \chi_2 & 2 & 0 & 0 & -2 & 2 & 0 & 0 \\
 \chi_i & 2 & i\sqrt2 & -i\sqrt2 & 0 & -2 & 0 & 0 \\
 \chi_{-i} & 2 & -i\sqrt2 & i\sqrt2 & 0 & -2 & 0 & 0
\end{array}
\end{equation*}
Inside $\cal F$, the class of $D^4$ becomes conjugate to the class of $S$, and class of $D^2$ becomes conjugate to the class of $DS$. None of the other conjugacy classes in $\SD_{16}$ are fused in $\cal F$. The $\cal F$-stable characters are therefore precisely the characters where the $4$'th value is equal to the $7$'th value, and the $5$'th is equal to the $6$'th.

Note that $\textbf 1$ is the only irreducible character for $\SD_{16}$ that is $\cal F$-stable.
Adding rows we see that the following four characters are also $\cal F$-stable: $\alpha:=\chi_a + \chi_i$, $\beta:=\chi_a+\chi_{-i}$, $\gamma:=\chi_b+\chi_2+\chi_i$ and $\delta:=\chi_b+\chi_2+\chi_{-i}$. Each of these four characters cannot be written as a sum of smaller $\cal F$-stable characters, so $\alpha$, $\beta$, $\gamma$ and $\delta$ correspond to representations that are irreducible in $R_+(\cal F)$.
At the same time, however, we have
\[\chi_a+\chi_b+\chi_2+\chi_i+\chi_{-i} = \alpha + \delta = \beta + \gamma.\]
Hence $\chi_a+\chi_b+\chi_2+\chi_i+\chi_{-i}$ corresponds to an element of $R_+(\cal F)$ that can be written as a sum of two irreducible elements in two different ways. The abelian monoid $R_+(\cal F)$ is therefore \emph{not} free, so the analogue of Theorem \ref{thmMonoidBasis} for complex representations is false.
\end{ex}

\makeatletter
\def\eprint#1{\@eprint#1 }
\def\@eprint #1:#2 {%
    \ifthenelse{\equal{#1}{arXiv}}%
        {\href{http://front.math.ucdavis.edu/#2}{arXiv:#2}}%
        {\href{#1:#2}{#1:#2}}%
}
\makeatother
\renewcommand{\PrintDOI}[1]{DOI \href{http://doi.org/#1}{#1}}

\end{document}